\documentclass[10pt,a4paper,reqno]{article}
\usepackage{amssymb,amsthm,enumerate,color, fullpage}
\usepackage{amsmath,fancybox,ascmac,amsfonts,bm}
\usepackage{mathrsfs}
\RequirePackage[dvipdfmx]{graphicx}
\usepackage{tikz}
\usetikzlibrary{intersections, calc}
\usetikzlibrary{decorations.markings}
\usetikzlibrary{arrows.meta}
\tikzset{
  arrows along my path/.style={
    postaction={
      decorate,
      decoration={
        markings,
        mark=between positions 0.03 and 1 step 10pt with {\arrow{Stealth[length=5pt]}},
   }}}}
\providecommand{\U}[1]{\protect\rule{.1in}{.1in}}
\theoremstyle{plain}
\newtheorem{thm}{Theorem}[section]

\theoremstyle{definition}
\newtheorem{defn}{Definition}
\theoremstyle{remark}
\newtheorem{rem}[thm]{Remark}
\newcommand{\C}{\mathbb{C}}

\newcommand{\N}{\mathbb{N}}

\newcommand{\cA}{\mathcal{A}}
\newcommand{\cB}{\mathcal{B}}
\newcommand{\cC}{\mathcal{C}}

\newcommand{\Dom}{{\sf Dom}}

\newcommand{\Tr}{{\sf Tr}}
\newcommand{\tr}{{\sf tr}}

\renewcommand{\leq}{\leqslant}
\renewcommand{\geq}{\geqslant}

\def\tr{{\rm tr}} 
\def\Tr{{\rm Tr}} 

\def\alg{\mathrm{alg}} 
\def\ideal{{\rm Ideal}} 
\def\Dom{{\rm D}} 
\def\ep{\varepsilon} 

\begin{document}
\title{Matrix models for cyclic monotone and monotone independences}
\author{Beno{\^i}t Collins \and Felix Leid \and Noriyoshi Sakuma}
\date{\today}
\maketitle

\begin{abstract}
Cyclic monotone independence is an algebraic notion of noncommutative independence, introduced in the study of multi-matrix random matrix models with small rank. Its algebraic form turns out to be surprisingly close  to monotone independence, which is why it was named cyclic monotone independence.
This paper conceptualizes this notion by showing that the same random matrix model is also a model for the monotone convergence with an appropriately chosen state. This observation provides a unified nonrandom matrix model for both types of monotone independences.
\end{abstract}

\section{Introduction}


Monotone independence was introduced by Muraki \cite{Muraki1996}, and Lu \cite{Lu1997} in the context of non-commutative probability theory.
Later on, Muraki \cite{Muraki1996, Muraki2001,Muraki2002,Muraki2003}, Hasebe \cite{Hasebe2010a,Hasebe2010b} and Hasebe and Saigo \cite{HS2011} developed monotone probability theory, which is a non-commutative probability theory with monotone independence, inspired by Voiculesu's free probability theory and Speicher's universal products \cite{Speicher1997}. The Construction of non-commutative probability spaces which realize monotone independence was achieved with the help of Fock spaces and universal products. This theory triggered substantial interest because monotone independence connects different subjects. For example,
Accardi, Ghorbal, and Obata \cite{AGO2004} realized monotone independence via the spectral analysis of the comb graph. Schleissinger \cite{Schleissinger2017} found a relation between monotone independence and SLE theory. The relation between Loewner chains and monotone probability theory is developing rapidly \cite{Jekel2020,FHS2020}. On the other hand, there was no natural (random) matrix model for monotone independence, unlike classical and free probability. 
The goal of this paper is to provide such a model.

Recently, motivated by the study of outliers in random matrix theory, Collins, Hasebe, and Sakuma developed in \cite{CHS2018} cyclic monotone independence. 
One cannot observe outliers from the empirical eigenvalue distributions of random matrices but their operator norm. To overcome this problem from the point of view of eigenvalue distributions, they proposed considering noncommutative probability spaces with a weight. The weight corresponds to the non-normalized trace. 
Computations of moments evidenced the notion of cyclic monotone independence -- a rule to compute joint moments that is quite 
similar to monotone convolution, with the additional property that it conserves traciality. 
This similarity was left in our previous paper as a curiosity to elucidate. However, it raised the natural question of the relation between both notions and the existence of a unified model for both convolutions. 

This paper is organized as follows. In Section \ref{section-2}, we recall the notions of monotone and cyclically monotone independences, and then state and prove a theoretical result about the structure of the free product algebra quotiented by the monotone (resp. cyclically monotone) free product state (Theorem \ref{main-abstract}).
In Section \ref{section-3}, we apply the result and provide a unified matrix model for monotonically (resp. free monotonically) free variables.
After that, we discuss random matrix models for monotone and cyclically monotone independences. 

\section{Notation and abstract result}\label{section-2}
Let us first review basic notations for monotone and cyclically monotone independences. For details, see \cite{CHS2018}.
A \emph{non-commutative measure space} is a pair $(\cA,\omega)$, where $\cA$ is a (unital or non-unital) $\ast$-algebra over $\C$. 
Let $\omega$ be a tracial weight, namely:  
\begin{itemize} 
\item $\omega$ is defined on a (possibly non-unital) $\ast$-subalgebra $\Dom(\omega)$ of $\cA$ and $\omega: \Dom(\omega)\to\C$ is linear,  
\item $\omega$ is positive, i.e.\ $\omega(a^* a) \geq0$ for every $a\in\Dom(\omega)$, 
\item $\omega(a^*)=\overline{\omega(a)}$ for all $a \in \Dom(\omega)$, 
\item $\omega(a b) = \omega(b a)$ for all $a,b \in\Dom(\omega)$. 
 \end{itemize}
Moreover, if $\cA$ is unital, $\Dom(\omega)=\cA$ and $\omega(1_\cA)=1$ then we call $(\cA,\omega)$ {\it a non-commutative probability space}.

Let  $(\cA,\omega)$ be a non-commutative measure space and let $a_1,\ldots ,a_k \in\Dom(\omega)$. 
The {\it distribution} of $(a_1,\ldots ,a_k)$ is the family of {\it (mixed) moments} 
\begin{align*}
\{\omega(a_{i_{1}}^{\ep_1}\dots a_{i_{p}}^{\ep_p}): p\geq1, 1 \leq i_{1},\dots,i_{p}\leq k, (\ep_1,\dots, \ep_p)\in\{1,\ast\}^p\}.
\end{align*}

Given non-commutative measure spaces $(\cA,\omega), (\cB,\xi)$ and elements $a_1, \dots, a_k\in\Dom(\omega)$, $b_1,\dots, b_k \in \Dom(\xi)$, we say that $(a_1,\dots, a_k)$ has {\it the same distribution} as $(b_1,\dots, b_k)$ if 
\begin{align}
 \omega(a_{i_1}^{\varepsilon_{1}}\cdots a_{i_p}^{\varepsilon_{p}})
 = \xi(b_{i_1}^{\varepsilon_{1}}\cdots b_{i_p}^{\varepsilon_{p}})
\end{align}
for any choice of $p\in\N$, $1\leq i_{1},\dots,i_{p}\leq k$ and $(\varepsilon_1,\dots, \varepsilon_p)\in\{1,\ast\}^p$. 



Let $(\cC,\tau,\omega)$ be a non-commutative probability space with a tracial weight $\tilde{\omega}$ (or $\omega$). 
Let $\cA, \cB$ be $\ast$-subalgebras of $\cC$ such that $1_\cC \in \cB$.
Let $\ideal_\cB(\cA)$ be the ideal generated by $\cA$ over $\cB$.
More precisely, 
\begin{align*}
\ideal_\cB(\cA):= \mathrm{span}\{b_0 a_1 b_1 \cdots a_n b_n: n \in \N, a_1,\dots, a_n \in \cA, b_0,\dots, b_n\in\cB\}, 
\end{align*}
which is a $*$-subalgebra of $\cC$ containing $\cA$.

We start with the definition of monotone independence.
\begin{defn}\label{construction} 
\begin{enumerate}[\rm(1)]
\item We say that the pair $(\cA,\cB)$  is {\it monotonically independent} 
with respect to $(\tau,\tilde{\omega})$ if 
\begin{itemize}
\item $\ideal_\cB(\cA) \subset \Dom(\tilde{\omega})$;  
\item for any $n \in \N, a_1,\dots, a_n \in \cA$ and any $b_1,\dots, b_n\in\cB$, we have that 
\begin{align*}
&\tilde{\omega}(b_{0}a_1 b_1a_2b_2 \cdots a_n b_n) =\tilde{\omega}(a_1a_2\cdots a_n) \tau(b_0)\tau(b_1)\tau(b_2)\cdots \tau(b_n),\\    
\end{align*}
\end{itemize}
\item Given $a_1,\dots, a_k \in \Dom(\omega)$ and $b_1,\dots, b_\ell \in \cC$, the pair $(\{a_1,\dots, a_k\},\{b_1,\dots, b_\ell\})$ is {\it monotone} 
if $(\alg\{a_1,\dots, a_k\}, \alg\{1_\cC,b_1,\dots, b_\ell\})$ is monotone.
Note that we {\it do not} assume that $\alg\{a_1,\dots, a_k\}$ contains the unit of $\cC$. 
\end{enumerate}
\end{defn}
Next, we recall the definition of cyclic monotone independence.
\begin{defn}\label{construction-cyclic} 
\begin{enumerate}[\rm(1)]
\item 
We say that the pair $(\cA,\cB)$  is 
{\it cyclically monotonically independent} 
with respect to $(\tau,\omega)$ if 
\begin{itemize}
\item $\ideal_\cB(\cA) \subset \Dom(\omega)$;  
\item for any $n \in \N, a_1,\dots, a_n \in \cA$ and any $b_1,\dots, b_n\in\cB$, we have that 
\begin{align*}
&\omega(b_{0}a_1 b_1a_2b_2 \cdots a_n b_n) =\omega(a_1a_2\cdots a_n) \tau(b_1)\tau(b_2)\cdots \tau(b_n b_{0}).  
\end{align*}
\end{itemize}
\item Given $a_1,\dots, a_k \in \Dom(\omega)$ and $b_1,\dots, b_\ell \in \cC$, the pair $(\{a_1,\dots, a_k\},\{b_1,\dots, b_\ell\})$ is 
{\it cyclically monotone} if $(\alg\{a_1,\dots, a_k\}, \alg\{1_\cC,b_1,\dots, b_\ell\})$ is 
cyclically monotone. 
As in the previous definition, we {\it do not} assume that $\alg\{a_1,\dots, a_k\}$ contains the unit of $\cC$. 
\end{enumerate}
\end{defn}

%
We add one more definition: given $H$ an arbitrary vector space, we call $\textrm{End}_{\text{fin}}(H)$ the collection of finite rank endomorphisms on $H$. As a vector space, it is canonically isomorphic to $H^*\otimes H$.

We are particularly interested in the case of $H=\cB$. In this case, $H^*\otimes H$ becomes $\textrm{End}_{\text{fin}}(\cB)$.

If $(\cB,\tau)$ is a unital noncommutative probability space such that 
\begin{align*}
    \cB=1\cdot\C\oplus \ker(\tau)\colon=1\cdot\C\oplus \mathring{\mathcal{B}}.
\end{align*}
we define $\psi_\tau$ on 
$\textrm{End}_{\text{fin}}(\cB)$
as the linear extension of
$$\psi_\tau (h_1^{\ast}\otimes h_2)=
h_1^*(1)\tau (h_2).$$
Intuitively, it is the upper left coefficient of the matrix of the endomorphism. $\psi_\tau$ is defined on $\textrm{End}_{\text{fin}}(\cB)$ and depends on $\tau$ however we will omit this dependence in the notation and write $\psi$.


Let $(\cA,\omega_0)$ be a (non unital) measure space and $(\cB,\tau)$ a noncommutative probability space such that 
\begin{align*}
    \cB=1\cdot\C\oplus \ker(\tau):=1\cdot\C\oplus \mathring{\mathcal{B}}.
\end{align*}
It means that, for any $b\in\mathcal{B}$, there exist $\mathring{b}\in \mathring{\mathcal{B}}$ such that $b = \mathring{b} + \tau(b)1$.

Consider
\begin{align*}
    \omega:=\omega_0\trianglerighteq\tau\colon \mathrm{D}(\omega)\to\C,\quad \tilde{\omega}:=\omega_0\triangleright\tau\colon\mathrm{D}(\tilde{\omega})\to\C,
\end{align*}
i.e., the cyclic monotone product and the monotone product of $\omega_{0}$ and $\tau$. 

Here we shall give a motivating example. Assume that $b\not=0$, $a\in\mathcal{A}$ and $a\not=0$.
We have
\begin{align*}
\omega(abab) 
&= \omega(a(\mathring{b} + \tau(b)1)a(\mathring{b} + \tau(b)1))
\\
&= \omega_{0}(a^{2})\tau(\mathring{b})^{2}
+2\tau(b)\omega_{0}(a^{2})\tau(\mathring{b})
+\tau(b)^{2}\omega_{0}(a^{2})
= \tau(b)^{2}\omega_{0}(a^{2})
\end{align*}
If $\tau(b)=0$ i.e. $b = \mathring{b}$ then $\omega(abab)=0$.
On the other hand, 
\begin{align*}
\omega(abba) 
&= \omega(a(\mathring{b} + \tau(b)1)(\mathring{b} + \tau(b)1)a)
\\
&= \omega_{0}(a^{2})\tau(\mathring{b}^{2})
+\tau(b)^{2}\omega_{0}(a^{2})
= \omega_{0}(a^{2})\tau(\mathring{b}^{2}) + \tau(b)^{2}\omega_{0}(a^{2})
\end{align*}
If $\tau(b)=0$ i.e. $b = \mathring{b}$ then $\omega(abba)=\omega_{0}(a^{2})\tau(\mathring{b}^{2})>0$.

We have the following rules:
for non-zero $a\in\mathcal{A}$ and $\mathring{b} \in \mathring{\mathcal{B}}$, we have

\begin{table}[htp]
\caption{Values of $\omega(P(a,\mathring{b})Q(a,\mathring{b}))$}
\begin{center}
\begin{tabular}{|c|c|c|c|c|c|}
\hline
$P(a,\mathring{b})\backslash Q(a,\mathring{b})$& $a$ & $a\mathring{b}$ & $\mathring{b}a$ & $\mathring{b}a\mathring{b}$ & \text{others}\\
\hline
$a$ & + & 0 & 0 & 0 & 0\\
\hline
$a\mathring{b}$ & 0 & 0 & + & 0 & 0\\
\hline
$\mathring{b}a$ & 0 & + & 0 & 0 & 0\\
\hline
$\mathring{b}a\mathring{b}$ & 0 & 0 & 0 & + & 0\\
\hline 
\text{others} & 0 & 0 & 0 & 0 & 0\\
\hline 
\end{tabular}
\end{center}
\label{table1}
\caption{Values of $\tilde\omega(P(a,\mathring{b})Q(a,\mathring{b}))$}
\begin{center}
\begin{tabular}{|c|c|c|c|c|c|}
\hline
$P(a,\mathring{b})\backslash Q(a,\mathring{b})$& $a$ & $a\mathring{b}$ & $\mathring{b}a$ & $\mathring{b}a\mathring{b}$ & \text{others}\\
\hline
$a$ & + & 0 & 0 & 0 & 0\\
\hline
$a\mathring{b}$ & 0 & 0 & + & 0 & 0\\
\hline
$\mathring{b}a$ & 0 & 0 & 0 & 0 & 0\\
\hline
$\mathring{b}a\mathring{b}$ & 0 & 0 & 0 & 0 & 0\\
\hline 
\text{others} & 0 & 0 & 0 & 0 & 0\\
\hline 
\end{tabular}
\end{center}
\label{table2}
\end{table}

From this observation, we can formulate the following.
First, we define
$I$ and $J$ through the following equations:
\begin{align*}
    I&:=\mathrm{D}(\omega_0)\oplus(\mathrm{D}(\omega_0)\otimes\mathring{\mathcal{B}})\oplus(\mathring{\mathcal{B}}\otimes\mathrm{D}(\omega_0))\oplus(\mathring{\mathcal{B}}\otimes\mathrm{D}(\omega_0)\otimes\mathring{\mathcal{B}})\oplus \underbrace{(\mathrm{D}(\omega_0)\otimes\mathring{\mathcal{B}}\otimes\mathrm{D}(\omega_0))\oplus\dots}_{J}\\
    J&:=(\mathrm{D}(\omega_0)\otimes\mathring{\mathcal{B}}\otimes\mathrm{D}(\omega_0))\oplus\dots
\end{align*}
Namely, $J$ is the sum of all tensor products with at least three legs and at least one $\mathring{\mathcal{B}}$, so that, in particular:
\begin{align*}    
    I=\mathrm{D}(\omega_0)\oplus(\mathrm{D}(\omega_0)\otimes\mathring{\mathcal{B}})\oplus(\mathring{\mathcal{B}}\otimes\mathrm{D}(\omega_0))\oplus(\mathring{\mathcal{B}}\otimes\mathrm{D}(\omega_0)\otimes\mathring{\mathcal{B}})\oplus J.
    \end{align*}
    Then $I$ is an ideal in the free product $*$-algebra $\cA * \cB$. Note that this free product does not amalgamate
    over the unit (in the first place, there is no unit in $\cA$). The $*$-algebra $\cA * \cB$ is unital.
    In addition, $J$ is an ideal in $I$.
    
\begin{thm}\label{main-abstract}
    The ideal $J$ annihilates $\omega$ and $\tilde{\omega}$, and we have a canonical map
    \begin{align*}
        \chi: I\to I/J\cong \cB\otimes \mathrm{D}(\omega_0) \otimes\cB \cong  \mathrm{D}( \omega_0)\otimes\mathrm{End}_{\mathrm{fin}}(\cB),
    \end{align*}
    which satisfies the following two properties:
    $$\omega\circ \chi = \omega_0\otimes \Tr$$
    and 
    $$\tilde \omega\circ \chi = \omega_0\otimes \psi$$

\end{thm}
\begin{rem}
About the cyclically monotone and monotone products $\trianglerighteq$, $\triangleright$, see, for details, page 1122 in \cite{CHS2018} and page 120 in \cite{Muraki2003}, respectively.
\end{rem}

\begin{proof}
First note that $\mathrm{D}(\omega) = \mathrm{D}(\tilde{\omega})=I$. We need to show that $J$ annihilates $\omega$ and $\tilde\omega$, take a monomial $b_0a_1b_1\dots a_n b_n=x\in J$, where $n>1$, then for any $b^\prime_0a^\prime_1b^\prime_1\dots a_m^\prime b_m^\prime=y\in I$ with $m>0$ we have
\begin{align*}
    \omega(xy)=\omega(yx)&=\omega_0(a_1\dots a_n a^\prime_1\dots a^\prime_m)\tau(b_0b^\prime_m)\tau(b_0^\prime b_n)\underbrace{\tau(b_1)}_{=0}\dots\tau(b_{n-1})\tau(b_1^\prime)\dots\tau(b_{m-1}^\prime)\\
    &=0.
\end{align*}
By linearity this extends to any elements $y\in I$ and $x\in J$. A similar calculation shows this for $\tilde\omega$.
Now clearly we have 
    \begin{align*}
        I/J&=\mathrm{D}(\omega_0)\oplus(\mathrm{D}(\omega_0)\otimes\mathring{\mathcal{B}})\oplus(\mathrm{D}(\omega_0)\otimes\mathring{\mathcal{B}})
        \oplus(\mathring{\mathcal{B}}\otimes\mathrm{D}(\omega_0)\otimes\mathring{\mathcal{B}}),
    \end{align*}
    the isomorphisms in the theorem are given by the identification
    \begin{align*}
        \mathrm{D}(\omega_0)\cong 1_B\otimes \mathrm{D}(\omega_0)\otimes 1_B, \mathrm{D}(\omega_0)\otimes \mathring{\mathcal{B}}\cong 1_B\otimes \mathrm{D}(\omega_0)\otimes \mathring{\mathcal{B}},\mathring{\mathcal{B}}\otimes \mathrm{D}(\omega_0)\cong\mathring{\mathcal{B}}\otimes \mathrm{D}(\omega_0)\otimes 1_B,
    \end{align*}
    and the usual identification $B\otimes B\cong\textrm{End}_{\textrm{fin}}(\cB)$. The diagram is due to the universal property of the quotient $I/J$.
\end{proof}

\section{Matrix model}\label{section-3}

This section is an application of the result of the previous section: it exhibits a matrix model for monotone independence and cyclically monotone independence. Basically, this is a concrete version of the above abstract result.

\subsection{Setup}

Let us denote by $M_\infty(\mathbb{C})$ the inductive limit given by the non unital embeddings $M_n(\mathbb{C})$ into $M_m(\mathbb{C})$ for $n<m$
\begin{align}\label{setup-inductive}
    f_{n,m}\colon M_n(\mathbb{C})\to M_m(\mathbb{C}),\quad a\mapsto\begin{pmatrix}
    a &0\\
    0 &0
    \end{pmatrix},
\end{align}
i.e. we are plugging $a \in M_{n}(\mathbb{C})$ into the left upper corner and padding it by zeros. Note that these embeddings are compatible with the non-normalized traces on $\Tr_n$ on $M_n(\mathbb{C})$, i.e. $\Tr_m\circ f_{m,n}=\Tr_n$. This induces the trace $\Tr$ on $M_{\infty}(\mathbb{C})$. Given any collection of elements $a_1,\dots,a_p\in M_\infty(\mathbb{C})$, we may always  choose $n\in\mathbb{N}$ big enough such that $a_i\in M_n(\mathbb{C})$, i.e. in an an upper corner of size $n$ of $M_{\infty}(\C)$.

From now on let $(\cC,\tau,\omega)$ be a non-commutative probability space with a tracial weight $\omega$, where $\cA, \cB$ be $\ast$-subalgebras of $\cC$ and we assume that $\mathcal{A}=M_\infty (\mathbb{C})$. Moreover we consider self-adjoint elements $a_{1},\ldots, a_{p}  \in \cA$  and $b_{0},\ldots, b_{q}\in\cB$ such that $b_{0}=1$, $\tau(b_{i}) = 0$ for any $i$, and,
$b_{i} \bot b_{j}$, that is,
$\tau(b_{i}b_{j}) = \delta_{ij}$ if $i\not=j$.

\begin{rem}
There is no loss in making these assumptions on $b_i$ because we can simultaneously take their real and imaginary parts if they are not self-adjoint. As for orthogonality, we can subsequently make a Gram-Schmidt orthogonalization to ensure that this property is satisfied too.
\end{rem}


We consider a non-commutative polynomial $P$, obtained as a sum of monomial that all contain at least one $a_{i}$. 
In other words, $P$ belongs to the two-sided ideal $\ideal_\cB(\cA)$. This condition comes from the models in Collins Hasebe and Sakuma \cite{CHS2018}.
Such polynomials are of the form:
\begin{equation}
P:=P(a_{1},\ldots, a_{p}, b_{0},\ldots, b_{q}) 
=  \sum_{i_{1},i_{2}=0}^{q} \sum_{j_{1}=1}^{p} \lambda_{i_{1},i_{2},j_{1}} b_{i_{1}}a_{j_{1}}b_{i_{2}}
+ \cdots ,
\end{equation}
where ``$\cdots$'' stands for an alternating sum of monomials in $a_i,b_j, j\ne 0$, with at least two $a$'s.

In the following we model the polynomial, we denote by $I_n\in M_n(\mathbb{C})$ the identity matrix. Then we consider the operator $\tilde{b}_{j}\in M_{n}(\C)\otimes M_{2^q}(\C)$, $\tilde{b}_{j}$, $j=0,1,\ldots q$, defined by

\begin{align}\label{def-onb-b}
\begin{split}
\tilde{b}_{0,n} &=I_n\otimes I_{2^q}=:I_n\otimes B_0\\
\tilde{b}_{1,n} &= 
I_n \otimes 
\underbrace{J\otimes I_{2} \otimes \dots \otimes I_{2}}_{=:B_1,\,  q\text{ terms}}\\
\tilde{b}_{2,n} &= 
I_n\otimes
\underbrace{I_{2}\otimes J \otimes \dots \otimes I_{2}}_{=:B_2,\,q\text{ terms}}\\
&\,\,\vdots\\
\tilde{b}_{q,n} &=
I_n\otimes
\underbrace{I_{2}\otimes \dots \otimes I_{2} \otimes J}_{=:B_q,\,q\text{ terms}},
\end{split}
\end{align}
where
$J=\begin{pmatrix}
0 &1\\
1 &0
\end{pmatrix}
$.
The tensor model for $P$ in $M_{n}(\C)\otimes M_{2^q}(\C)$ is the matrix 
\begin{align*}
\tilde P:=\tilde P(a_{1},\ldots, a_{p}, \tilde b_{0,n},\ldots \tilde b_{q,n})
=  \sum_{i_{1},i_{2},j_{1}} \lambda_{i_{1},i_{2},j_{1}}\tilde b_{i_{1},n}\cdot \phi (a_{j_{1}})\cdot \tilde b_{i_{2},n},
\end{align*}
where 
\begin{align*}
E_{11}=\begin{pmatrix}
1 &0\\
0 &0
\end{pmatrix}
\text{ and }
\phi(a) := a \otimes E_{11}^{\otimes q}
=
\begin{pmatrix}
a & O & \dots O\\
O & O & \dots O\\
O & O & \dots O
\end{pmatrix}
\in M_{n}(\C)\otimes M_{2^q}(\C)
\end{align*}
 and we assumed $n\in\mathbb{N}$ large enough to fit $a_i\in M_n(\mathbb{C})$ as described earlier.

\subsection{A tensor model for cyclically monotone independence}

Now, we assume that
the pair $(\cA,\cB)$  is cyclically monotone with respect to $(\tau,\omega)$.

Then the moments of $P$ under $\omega$ can be obtained from the tensor model $\tilde P$ in the following sense.
\begin{thm}\label{thm-cyclic}
We have $\omega(P(a_{1},\ldots, a_{p}, b_{0},\ldots, b_{q})^{k}) 
= 
\mathrm{Tr}_{n}\otimes \mathrm{Tr}_{2}^{\otimes q}(\tilde P(a_{1},\ldots, a_{p}, \tilde b_{0,n},\ldots, \tilde b_{q,n})^{k})$.
\end{thm}

\begin{proof}

Let us start by computing the $k$-moment of $P$. Recall that by the main theorem, only the terms containing one $a$ contribute to the calculation
\begin{align*}
    \omega(P^k)&=\omega\left(\left(
    \sum_{i_{1},i_{2}=0}^{q} \sum_{j_{1}=1}^{p} \lambda_{i_1,i_2,j_1} b_{i_1} a_{j_1} b_{i_{2}}+\cdots
    \right)^k\right)\\
    &= \sum_{i_{1},i_{2},\ldots,i_{2k}=0}^{q} \sum_{j_{1},\ldots, j_{k}=1}^{p} 
    \left(
    \prod_{r=1}^{k}\lambda_{i_{2r-1},i_{2r},j_{r}}
    \right)
    \omega(b_{i_1} a_{j_1} b_{i_{2}}\dots b_{i_{2k-1}} a_{j_k} b_{i_{2k}})\\
    &= \sum_{i_{1},i_{2},\ldots,i_{2k}=0}^{q} \sum_{j_{1},\ldots, j_{k}=1}^{p}  
    \left(
    \prod_{r=1}^{k}\lambda_{i_{2r-1},i_{2r},j_{r}}
    \right)
    \omega(a_{j_1} \dots  a_{j_k} ) \tau(b_{i_{2k}}b_{i_{1}})\tau(b_{i_{2}}b_{i_{3}})
    \dots \tau(b_{i_{2k-2}}b_{i_{2k-1}})\\
    &= \sum_{i_{1},i_{2},\ldots,i_{2k}=0}^{q} \sum_{j_{1},\ldots, j_{k}=1}^{p} 
    \left(
    \prod_{r=1}^{k}\lambda_{i_{2r-1},i_{2r},j_{r}}
    \right)
    \omega(a_{j_1} \dots  a_{j_k} ) \delta_{i_{2k},i_{1}} \delta_{i_{2},i_{3}}
    \dots \delta_{i_{2k-2},i_{2k-1}}\\
    &= \sum_{i_1,i_3,\ldots,i_{2k-1}=0}^{q}\sum_{j_{1},\ldots,j_{k}=1}^{p} 
    \left(
    \prod_{r=1}^{k}\lambda_{i_{2r-1},i_{2r+1},j_{r}}
    \right)
    \omega(a_{j_1} \dots  a_{j_k} ).
\end{align*}

We show that it is equal to $ \mathrm{Tr}_{n}\otimes \mathrm{Tr}_{2}^{\otimes q}(\tilde P(a_{1},\ldots, a_{p}, b_{0},\ldots, b_{q})^{k})$.
We have
\begin{align*}
    \mathrm{Tr}_{2}^{\otimes q} (B_{i_1}E_{11}^{\otimes q} B_{i_{2}}\dots B_{i_{2k-1}} E_{11}^{\otimes q}B_{i_{2k}})
    =\begin{cases}
    1 & i_{2k} = i_{1},\ldots, i_{2k-2} = i_{2k-1}\\
    0 &\text{otherwise}
    \end{cases}
\end{align*}
from $J E_{11} J = \begin{pmatrix}0& 0\\0 & 1 \end{pmatrix}$ and $J E_{11} J E_{11} = 0$.
For example, if we consider the $q=2$ case, we have 
\begin{align*}
    &\mathrm{Tr}_{2}^{\otimes 2} (B_{1}E_{11}^{\otimes 2} B_{2} B_{2} E_{11}^{\otimes 2}B_{1})
    =\mathrm{Tr}_{2}(J E_{11} J) \mathrm{Tr}_2(E_{11}) = 1, \\
    &\mathrm{Tr}_{2}^{\otimes 2} (B_{1}E_{11}^{\otimes 2} B_{2} B_{1} E_{11}^{\otimes 2}B_{2})
    =\mathrm{Tr}_{2}(JE_{11}JE_{11}) \mathrm{Tr}_{2}(E_{11}JE_{11}J) = 0.
\end{align*}
Thus, we obtain
\begin{align*}
    &\mathrm{Tr}_{n}\otimes \mathrm{Tr}_{2}^{\otimes q}(\tilde P(a_{1},\ldots, a_{p}, \tilde b_{0},\ldots, \tilde b_{q})^{k})\\
    &=\sum_{i_{1},i_{2},\ldots,i_{2k}=0}^{q} \sum_{j_{1},\ldots, j_{k}=1}^{p}  
    \left(
    \prod_{r=1}^{k}\lambda_{i_{2r-1},i_{2r},j_{r}}
    \right)
    \mathrm{Tr}_{n}(a_{j_1} \dots  a_{j_k} )
    \mathrm{Tr}_{2}^{\otimes q} (B_{i_1}E_{11}^{\otimes q} B_{i_{2}}\dots B_{i_{2k-1}} E_{11}^{\otimes q} B_{i_{2k}})\\
    &=\sum_{i_{1},i_{2},\ldots,i_{2k}=0}^{q} \sum_{j_{1},\ldots, j_{k}=1}^{p}  
    \left(
    \prod_{r=1}^{k}\lambda_{i_{2r-1},i_{2r},j_{r}}
    \right)
    \mathrm{Tr}_{n}(a_{j_1} \dots  a_{j_k} ) \delta_{i_{2k},i_{1}} \delta_{i_{2},i_{3}}
    \dots \delta_{i_{2k-2},i_{2k-1}}\\
    &=\sum_{i_1,i_3,\ldots,i_{2k-1}=0}^{q} \sum_{j_{1},\ldots,j_{k}=1}^{p}  
    \left(
    \prod_{r=1}^{k}\lambda_{i_{2r-1},i_{2r+1},j_{r}}
    \right)
    \mathrm{Tr}_{n}(a_{j_1} \dots  a_{j_k} )\\
    &= \sum_{i_1,i_3,\ldots,i_{2k-1}=0}^{q} \sum_{j_{1},\ldots,j_{k}=1}^{p} 
    \left(
    \prod_{r=1}^{k}\lambda_{i_{2r-1},i_{2r+1},j_{r}}
    \right)
    \omega(a_{j_1} \dots  a_{j_k} ).
\end{align*}
Recall that we have $\mathrm{Tr}_{n}(a_{j_1} \dots  a_{j_k} ) = \omega(a_{j_1} \dots  a_{j_k} )$, which concludes the proof.
\end{proof}

\subsection{A tensor model for monotone independence}

As it follows from the main theorem, we can also treat monotone independence with the very same model,
provided that we modify the state. 
Let us assume that the pair $(\cA,\cB)$  is monotone independent with respect to $(\tau,\tilde{\omega})$.
\begin{thm}\label{thm-monotone}
Let 
$\eta (B)=b_{11}$ for $B=(B_{ij})_{ij} \in M_2(\C)$.
We have $\tilde{\omega}(P(a_{1},\ldots, a_{p}, b_{0},\ldots, b_{q})^{k}) = \mathrm{Tr}_{n}\otimes \eta^{\otimes q}(\tilde P(a_{1},\ldots, a_{p}, b_{0},\ldots, b_{q})^{k})$.
\end{thm}

\begin{proof}
First, note that similar to the cyclically monotone case, we may omit all terms containing at least two $a$'s. By monotone independence, we have for the remaining terms:
\begin{align*}
    &\tilde{\omega}(P^k) 
    = \sum_{i_{1},i_{2},\ldots,i_{2k}=0}^{q} \sum_{j_{1},\ldots, j_{k}=1}^{p}  
    \left(\prod_{r=1}^{k}\lambda_{i_{2r-1},i_{2r},j_{r}}\right)
    \tilde{\omega}(a_{j_1} \dots  a_{j_k} ) \tau(b_{i_{2k}})\tau(b_{i_{1}})\tau(b_{i_{2}}b_{i_{3}})
    \dots \tau(b_{i_{2k-2}}b_{i_{2k-1}})\\
    &= \sum_{i_{1},i_{2},\ldots,i_{2k}=0}^{q} \sum_{j_{1},\ldots, j_{k}=1}^{p} 
    \left(\prod_{r=1}^{k}\lambda_{i_{2r-1},i_{2r},j_{r}}\right)
    \tilde{\omega}(a_{j_1} \dots  a_{j_k} ) 
    \delta_{i_{2k},0}\delta_{i_{1},0}
    \delta_{i_{2},i_{3}} \dots \delta_{i_{2k-2},i_{2k-1}}\\
    &= \sum_{i_2,i_4,\ldots,i_{2k-2}=0}^{q}\sum_{j_{1},\ldots,j_{k}=1}^{p} 
    \lambda_{0,i_{2},j_{1}}
    \left(\prod_{r=2}^{k-1}\lambda_{i_{2r-2},i_{2r},j_{r}}\right)
    \lambda_{i_{2k-2},0,j_{k}}
    \tilde{\omega}(a_{j_1} \dots  a_{j_k} ).
\end{align*}
On the other hand, since $\eta(J E_{11} J) = 0$,
we obtain 
\begin{align*}
    &\mathrm{Tr}_{n}\otimes \eta^{\otimes q}(\tilde P^k)\\ 
    &= \sum_{i_{1},i_{2},\ldots,i_{2k}=0}^{q} \sum_{j_{1},\ldots, j_{k}=1}^{p} 
    \left(\prod_{r=1}^{k}\lambda_{i_{2r-1},i_{2r},j_{r}}\right)
    \mathrm{Tr}_{n}(a_{j_1} \dots  a_{j_k} )
    \eta^{\otimes q}
    (B_{i_{1}} E_{11}^{\otimes q} B_{i_{2}}\cdots 
    B_{i_{2k-1}} E_{11}^{\otimes q} B_{i_{2k}})\\
    &= \sum_{i_{1},i_{2},\ldots,i_{2k}=0}^{q} \sum_{j_{1},\ldots, j_{k}=1}^{p}  
    \left(\prod_{r=1}^{k}\lambda_{i_{2r-1},i_{2r},j_{r}}\right)
    \mathrm{Tr}_{n}(a_{j_1} \dots  a_{j_k} )
    \delta_{i_{1}0}
    \delta_{i_{2}i_{3}}
    \dots
    \delta_{i_{2k-2}i_{2k-1}}
    \delta_{i_{2k}0}
    \\
    &= \sum_{i_2,i_4,\ldots,i_{2k-2}=0}^{q}\sum_{j_{1},\ldots,j_{k}=1}^{p} 
    \lambda_{0,i_{2},j_{1}}
    \left(\prod_{r=2}^{k-1}\lambda_{i_{2r-2},i_{2r},j_{r}}\right)
    \lambda_{i_{2k-2},0,j_{k}}
    \tilde{\omega}(a_{j_1} \dots  a_{j_k} ).
\end{align*}

\end{proof}

\subsection{Replacing tensors by limit swaps}

This subsection is a simple observation: the previous proofs rely on the same model
that relies on tensors and considers two different states -- one for monotonically cyclic
independence, and one for monotone independence. 
Here, we show that in a context of a \emph{sequence} of matrix models, we can
avoid resorting to tensors. 

Let us first recall that we are interested in polynomials $\tilde P\in M_n(\C)\otimes M_2(\C)^{\otimes q}$.
In addition, $\tilde P$ depends tacitly on $n,q$ and is 
well defined for any $q,n$ large enough via the embedding described in Equation \eqref{setup-inductive}.
Likewise, with the same embedding, we 
and we will freely view $\tilde P$ (as a double sequence
in $n,q$ as elements of $M_\infty(\C)$.
In the sequel of this paper, we denote by $\eta_l$ the function
\begin{align*}
    \mathcal{A}=M_\infty(\C)\to\C,\quad x\mapsto \sum_{k=0}^l x_{kk}.
\end{align*}
This is sometimes called a partial trace, e.g. in the context of Horn inequalities 
(although it is not the partial trace of quantum information theory). We have 

\begin{thm}\label{thm-limits}
The following holds true
\begin{align*}
    \omega (P^k)= \lim_{n\to\infty}\lim_{l\to\infty}\eta_l (\tilde P^k),
\end{align*}
i.e. convergence to the cyclically monotone independent moments and
\begin{align*}
    \tilde\omega (P^k)=\lim_{l\to\infty}\lim_{n\to\infty}\eta_l (\tilde P^k),
\end{align*}
i.e. convergence to the monotone independent moments.
\end{thm}
\begin{proof}
First, recall that we have

\begin{align*}\omega(P(a_{1},\ldots, a_{p}, b_{0},\ldots, b_{q})^{k}) 
= \mathrm{Tr}(\tilde P(a_{1},\ldots, a_{p}, \tilde b_{0,n},\ldots, \tilde b_{q,n})^{k})
\end{align*}
by Theorem \ref{thm-cyclic}. Note that the right hand side has a dependency on the size $n$ that can be easily removed by letting $n\to\infty$ (taking $n$ large enough is sufficient in the proof), 
and we get
$$\omega(P(a_{1},\ldots, a_{p}, b_{0},\ldots, b_{q})^{k}) 
= \lim_{n\to\infty}\mathrm{Tr}(\tilde P(a_{1},\ldots, a_{p}, \tilde b_{0,n},\ldots, \tilde b_{q,n})^{k})$$
which proves the first claim.

On the other hand we note $\mathrm{Tr}=\lim_{l\to\infty} \eta_l$, and therefore we get
$$\omega(P(a_{1},\ldots, a_{p}, b_{0},\ldots, b_{q})^{k}) 
= \lim_{n\to\infty}\lim_{l\to\infty} \eta_l(\tilde P(a_{1},\ldots, a_{p}, \tilde b_{0,n},\ldots, \tilde b_{q,n})^{k})$$
Likewise, theorem \ref{thm-monotone} gives 
$$\tilde\omega(P(a_{1},\ldots, a_{p}, b_{0},\ldots, b_{q})^{k}) 
= \lim_{n\to\infty}\mathrm{Tr}(I_n \tilde P(a_{1},\ldots, a_{p}, \tilde b_{0,n},\ldots, \tilde b_{q,n})^{k})$$
Rewriting it as
$$\tilde\omega(P(a_{1},\ldots, a_{p}, b_{0},\ldots, b_{q})^{k}) 
= \mathrm{Tr}( \lim_{n\to\infty} I_n \tilde P(a_{1},\ldots, a_{p}, \tilde b_{0,n},\ldots, \tilde b_{q,n})^{k}),$$
we get
$$\tilde\omega(P(a_{1},\ldots, a_{p}, b_{0},\ldots, b_{q})^{k}) 
= \lim_{l\to\infty} \eta_l( \lim_{n\to\infty} I_n \tilde P(a_{1},\ldots, a_{p}, \tilde b_{0,n},\ldots, \tilde b_{q,n})^{k}),$$
but clearly, for $n\ge l$, we have $\eta_l (I_nP)=\eta_l (P)$ and therefore 
$$\tilde\omega(P(a_{1},\ldots, a_{p}, b_{0},\ldots, b_{q})^{k}) 
= \lim_{l\to\infty} \eta_l( \lim_{n\to\infty}  \tilde P(a_{1},\ldots, a_{p}, \tilde b_{0,n},\ldots, \tilde b_{q,n})^{k}),$$
which concludes the proof.
\end{proof}

\subsection{Example}
Let us illustrate our result with an example.
We consider a non-commutative probability space $(\mathcal{C},\tau,\omega)$ with a tracial weight $\omega$,
finite rank operator $a\in \mathrm{D}(\omega)$ with the eigenvalues $(2^{-1},2^{-2},2^{-3})$ and a operator $b\in\mathcal{B}$ with $\tau(b)=0$ and $\tau(b^2)=1$. 
We, in addition, assume that the operators $a$ and $b$ are cyclically monotone independent.

From our result we can give the following matrix model $A$ and $B$ for $a$ and $b$: 
\begin{align*}
    A=\mathrm{diag}{(2^{-1},2^{-2},2^{-3})}
    \otimes E_{11} 
    = 
    \left(
		\begin{array}{cccccc}
		\frac{1}{2} & 0 & 0 & 0 & 0 & 0 \\
		0 & \frac{1}{4} & 0 & 0 & 0 & 0 \\
		0 & 0 & \frac{1}{8} & 0 & 0 & 0 \\
		0 & 0 & 0 & 0 & 0 & 0 \\
		0 & 0 & 0 & 0 & 0 & 0 \\
		0 & 0 & 0 & 0 & 0 & 0 \\
		\end{array}
	\right),\quad
    B=I_{3} \otimes J
    = \left(
    	\begin{array}{cccccc}
		0 & 0 & 0 & 1 & 0 & 0 \\
		0 & 0 & 0 & 0 & 1 & 0 \\
		0 & 0 & 0 & 0 & 0 & 1 \\
		1 & 0 & 0 & 0 & 0 & 0 \\
		0 & 1 & 0 & 0 & 0 & 0 \\
		0 & 0 & 1 & 0 & 0 & 0 \\
		\end{array}
	\right). 
\end{align*}
We consider $X = A + BAB$ and $Y = AB + BA$:
\begin{align*}
X= 
\left(
\begin{array}{cccccc}
 \frac{1}{2} & 0 & 0 & 0 & 0 & 0 \\
 0 & \frac{1}{4} & 0 & 0 & 0 & 0 \\
 0 & 0 & \frac{1}{8} & 0 & 0 & 0 \\
 0 & 0 & 0 & \frac{1}{2} & 0 & 0 \\
 0 & 0 & 0 & 0 & \frac{1}{4} & 0 \\
 0 & 0 & 0 & 0 & 0 & \frac{1}{8} \\
\end{array}
\right),
\quad
Y=
\left(
\begin{array}{cccccc}
 0 & 0 & 0 & \frac{1}{2} & 0 & 0 \\
 0 & 0 & 0 & 0 & \frac{1}{4} & 0 \\
 0 & 0 & 0 & 0 & 0 & \frac{1}{8} \\
 \frac{1}{2} & 0 & 0 & 0 & 0 & 0 \\
 0 & \frac{1}{4} & 0 & 0 & 0 & 0 \\
 0 & 0 & \frac{1}{8} & 0 & 0 & 0 \\
\end{array}
\right)
\end{align*}
The lists of eigenvalues $X$ and $Y$ are 
\begin{align*}
    \mathrm{EV}(X) =\{ 2^{-1},2^{-1},2^{-2},2^{-2},2^{-3},2^{-3}\}
    \text{ and }
    \mathrm{EV}(Y) = \{ 2^{-1},-2^{-1},2^{-2},-2^{-2},2^{-3},-2^{-3}\},
\end{align*}
where $\mathrm{EV}(X)$ is the multi-set of eigenvalues of $X$.
They correspond with the eigenvalues of $a + bab$ and $ab + ba$, respectively.

\subsection{Random matrix models and concluding remarks}

This paper presents a matrix model for monotone and cyclically monotone independences, which is not random. 
In free probability, there exist random matrix models which are not random 
but random matrix models are much more common. 
Therefore, it is natural to wonder whether there is a random model in the case of monotone and cyclically monotone independences.
This turns out to be the case, and we can easily show that the model introduced by
Collins, Hasebe, and Sakuma in \cite{CHS2018}
is also a model for monotone independence, provided that
we consider $ \lim_l\lim_n\eta_l$ as our limiting state. 

\begin{thm}\label{thm-random}
Let $A_1^{(n)},\ldots , A_p^{(n)}, B_1^{(n)},\ldots , B_q^{(n)}\in M_n(\C)$ be matrices such that there is $C>0$ such that for any $m\in\N$, $i_1,\dots, i_m \in \{1,\dots,p\}$ and $j_1,\dots, j_m\in \{1,\dots, q\}$ we have
$$|\Tr (A_{i_1}^{(n)}\ldots A_{i_m}^{(n)})|\le C$$
and
$$|\tr (B_{j_1}^{(n)}\ldots B_{j_m}^{(n)})|\le C.$$ Moreover let $U=U(n)$ be a Haar unitary random matrix. Then
\begin{align*}
    \lim_{l\to\infty}\lim_{n\to\infty}\eta_l(UB_{i_0}^{(n)}U^*
A_{i_1}^{(n)} UB_{i_1}^{(n)}U^*\ldots A_{i_m}^{(n)} UB_{i_m}^{(n)}U^*))=\lim_{l\to\infty}\lim_{n\to\infty}\eta_l(A_{i_1}^{(n)} \ldots A_{i_m}^{(n)})\tr (B_{i_1}^{(n)})\ldots \tr (B_{i_m}^{(n)}).
\end{align*}
\end{thm}

\begin{proof}
We have that
$$
    |\Tr(A_{i_1}^{(n)} UB_{i_1}^{(n)}U^*\ldots A_{i_m}^{(n)} UB_{i_m}^{(n)}U^*)
    -\Tr (A_{i_1}^{(n)}\ldots A_{i_1}^{(n)})\tr (B_{i_1}^{(n)})\ldots \tr (B_{i_m}^{(n)})|
    =O(n^{-1})
$$
For $l\le n$, calling $I_l$ the matrix whose first $l$ diagonal entries are $1$
and all other entries in $M_n(\C)$ are zero,
the above formula implies
$$|\Tr( I_l UB_{i_0}^{(n)}U^*
A_{i_1}^{(n)} UB_{i_1}^{(n)}U^*\ldots A_{i_m}^{(n)} UB_{i_m}^{(n)}U^*)
-\Tr (I_l A_{i_1}^{(n)}\ldots A_{i_1}^{(n)})\tr (B_{i_0}^{(n)})\ldots \tr (B_{i_m}^{(n)})|=O(n^{-1}).$$
Noting that
$$
\Tr (I_l A_{i_1}^{(n)}\ldots A_{i_m}^{(n)})=
\eta_l (A_{i_1}^{(n)}\ldots A_{i_m}^{(n)}),$$
we obtain our desired model for monotone convergence
by letting $n\to\infty$ followed by $l\to\infty$.
In the space of compact operators of $l^2$, this is compared
to the result of Collins Hasebe Sakuma for the very same model,
where we first take $l\to\infty$ (to get the non-normalized
trace), followed by $n\to\infty$ (to get monotone convergence). 
\end{proof}

Let us discuss the relation between this result and previous results on this model.
In the paper \cite{CHS2018}, we considered the same model, with in addition an assumption of moment convergence for the sequences 
$A_1^{(n)},\ldots , A_p^{(n)}, B_1^{(n)},\ldots $ and $B_q^{(n)}\in M_n(\C)$ respectively.
Namely, in addition to assuming 
$|\Tr (A_{i_1}^{(n)}\ldots A_{i_m}^{(n)})|\le C$
and
$|\tr (B_{j_1}^{(n)}\ldots B_{j_m}^{(n)})|\le C,$
we assumed that 
$$\lim_n\Tr (A_{i_1}^{(n)}\ldots A_{i_m}^{(n)})=f(i_1,\ldots, i_m)
{\rm  \,\, and \,\, }
\lim_n\tr (B_{j_1}^{(n)}\ldots B_{j_m}^{(n)})=g(j_1,\ldots ,j_m).$$
There, our main result was to prove that
$$\lim_n\Tr(A_{i_1}^{(n)} UB_{i_1}^{(n)}U^*\ldots A_{i_m}^{(n)} UB_{i_m}^{(n)}U^*)$$
converges almost surely to
$$\lim_n \Tr(A_{i_1}^{(n)} \ldots A_{i_m}^{(n)})\tr (B_{i_1}^{(n)})\ldots \tr (B_{i_m}^{(n)}),$$
which defines the cyclic monotone independence. 
Given that $\Tr=\lim_l \eta_l$, the above theorem 
implies that the monotone state can be obtained
from the same model provided that we swap the limits
$n$ and $l$.
In this respect, we completely generalize and conceptualize the results of \cite{CHS2018}.

We conclude by noting, we were informed by Takahiro Hasebe that in a work in preparation with Octavio Arizmendi and Franz Lehner, among others, they 
obtain models for cyclically monotone independence and monotone independence, which are different from ours. 

\emph{Acknowledgments:} We are grateful to Roland Speicher and Takahiro Hasebe for preliminary comments on our preprint and useful discussions, and to Akihiro Miyagawa for a careful reading. 
BC was partially supported by JSPS Kakenhi 17H04823, 20K20882, 21H00987, JPJSBP120203202.
FL was supported by the SFB-TRR 195 'Symbolic Tools in Mathematics and their Application' of the German Research Foundation (DFG).
NS was partially supported by JSPS Kakenhi 19H01791, 19K03515, JPJSBP120209921, JPJSBP120203202.



\begin{thebibliography}{10}
\bibitem{AGO2004}
L. Accardi, A. Ben Ghorbal and N. Obata,
Monotone independence, comb graphs, and Bose-Einstein condensation. 
{\it Infin. Dimens. Anal. Quantum Probab. Relat. Top.}, {\bf 7}, (2004), no. 3, 419--435. 

\bibitem{CHS2018}
B. Collins, T. Hasebe and N.Sakuma,
Free probability for purely discrete eigenvalues of random matrices.
{\it J. Math. Soc. Japan}, {\bf 70}, (2018), no. 3, 1111--1150.

\bibitem{FHS2020}
U. Franz, T. Hasebe and S. Schleissinger,
Monotone increment processes, classical Markov processes, and Loewner chains. 
{\it Dissertationes Math}., {\bf 552}, (2020), 119 pp.

\bibitem{Hasebe2010a}
T. Hasebe,
Monotone convolution and monotone infinite divisibility from complex analytic viewpoint. {\it Infin. Dimens. Anal. Quantum Probab. Relat. Top.}, {\bf 13}, (2010), no. 1, 111--131. 

\bibitem{Hasebe2010b}
T. Hasebe,
Monotone convolution semigroups. 
{\it Studia Math.}, {\bf 200}, (2010), no. 2, 175--199.

\bibitem{Jekel2020}
D. Jekel,
Operator-valued chordal Loewner chains and non-commutative probability. 
{\it J. Funct. Anal.}, {\bf 278}, (2020), no. 10, 108452, 100 pp. 
\bibitem{JekelLiu2020}
D. Jekel and W. Liu,
An operad of non-commutative independences defined by trees. {\it Dissertationes Math.}, 553 (2020), 100 pp.

\bibitem{HS2011}
T. Hasebe and H. Saigo,
The monotone cumulants. 
{\it Ann. Inst. Henri Poincar{\'e} Probab. Stat.}, {\bf 47}, (2011), no. 4, 1160--1170.

\bibitem{Lenczewski2010}
R. Lenczewski,
Matricially free random variables.
{\it J. Funct. Anal.}, {\bf 258}, (2010), no. 12, 4075--4121.

\bibitem{Lu1997}
Y. G. Lu,
An interacting free Fock space and the arcsine law.
{\it Probab. Math. Statist.}, {\bf 17}, (1997), no. 1, Acta Univ. Wratislav. No. 1928, 149--166.
\bibitem{Muraki1996}
N. Muraki,
A new example of noncommutative "de Moivre-Laplace theorem''. 
{\it Probability theory and mathematical statistics (Tokyo, 1995)}, 353--362, World Sci. Publ., River Edge, NJ, 1996.
\bibitem{Muraki2001}
N. Muraki,
Monotonic independence, monotonic central limit theorem and monotonic law of small numbers. 
{\it Infin. Dimens. Anal. Quantum Probab. Relat. Top}. {\bf 4}, (2001), no. 1, 39--58.
\bibitem{Muraki2002}
N. Muraki,
The five independences as quasi-universal products. 
{\it Infin. Dimens. Anal. Quantum Probab. Relat. Top.}, {\bf 5}, (2002), no. 1, 113--134. 
\bibitem{Muraki2003}
N. Muraki,
The five independences as natural products. 
{\it Infin. Dimens. Anal. Quantum Probab. Relat. Top.}, {\bf 6}, (2003), no. 3, 337--371.

\bibitem{Schleissinger2017}
S. Schleissinger,
The chordal Loewner equation and monotone probability theory. 
{\it Infin. Dimens. Anal. Quantum Probab. Relat. Top.}, {\bf 20}, (2017), no. 3, 1750016, 17 pp.

\bibitem{Speicher1997}
R. Speicher, 
\newblock On universal products.
\newblock{\em Free probability theory (Waterloo, ON, 1995), 257--266, Fields Inst. Commun.}, {\bf 12}, Amer. Math. Soc., Providence, RI, 1997.
\end{thebibliography}
\end{document}